\documentclass[a4paper,twoside,english]{amsart}
\usepackage[T1]{fontenc}
\usepackage[utf8]{inputenc}
\usepackage{babel}
\usepackage{verbatim}
\usepackage{amsthm}
\usepackage{amsmath}
\usepackage{amssymb}
\usepackage{a4wide}
\usepackage{tikz-cd}
\usepackage[normalem]{ulem}
\usepackage{thmtools}

\usepackage[dvipsnames]{xcolor}

\pdfpageheight\paperheight 
\pdfpagewidth\paperwidth

\theoremstyle{plain}
    \newtheorem{theorem}{Theorem}[section]
    \newtheorem{lemma}[theorem]{Lemma}
    \newtheorem{proposition}[theorem]{Proposition}
    \newtheorem{corollary}[theorem]{Corollary}
    \newtheorem{algorithm}[theorem]{Algorithm}

\theoremstyle{definition}
    \newtheorem{remark}[theorem]{Remark}
    \newtheorem{problem}[theorem]{Problem}
    \newtheorem{example}[theorem]{Example}
    
    \newtheorem{notation}[theorem]{Notation}

\usepackage[
    bookmarksnumbered=true,
    bookmarksopen=false,
    breaklinks=false,
    pdfborder={0 0 1},
    backref=false,
    colorlinks=false]{hyperref}
\hypersetup{
    pdfauthor={
        A.~Gallese,
        D.~Lombardo,
        F.~Naccarato,
        and U.~Zannier}}

\usepackage{cleveref}
    \crefname{lemma}{lemma}{lemmas}
    \Crefname{lemma}{Lemma}{Lemmas}
    \crefname{proposition}{proposition}{propositions}
    \Crefname{proposition}{Proposition}{Propositions}
    \crefname{corollary}{corollary}{corollaries}
    \Crefname{corollary}{Corollary}{Corollaries}
    \crefname{definition}{definition}{definitions}
    \Crefname{definition}{Definition}{Definitions}
    \crefname{remark}{remark}{remarks}
    \Crefname{remark}{Remark}{Remarks}
    \crefname{notation}{notation}{notations}
    \Crefname{notation}{Notation}{Notations}
    \crefname{algorithm}{algorithm}{algorithms}
    \Crefname{algorithm}{Algorithm}{Algorithms}
    \crefname{example}{example}{examples}
    \Crefname{example}{Example}{Examples}

\usepackage[left=3.6cm,right=3.9cm,top=3cm,bottom=3cm]{geometry}

\def\CVD{{\hfill\hfil{\lower 2pt\hbox{\vrule\vbox to 7pt
{\hrule width  4pt\varphifill\hrule}\varphirule}}}\par}

\title[Finding the complement of an elliptic curve inside a Jacobian]{Finding the complement of an elliptic curve \\ inside a Jacobian}

\author[Gallese]{Andrea Gallese}
\address{Scuola Normale Superiore, Piazza dei Cavalieri 7, 56126 Pisa, Italy}
\email{andrea.gallese@sns.it}

\author[Lombardo]{Davide Lombardo}
\address{Department of mathematics, University of Pisa, Italy}
\email{davide.lombardo@unipi.it}

\author[Naccarato]{Francesco Naccarato}
\address{D-MATH, ETH Z\"urich, Switzerland}
\email{francesco.naccarato@math.ethz.ch}

\author[Zannier]{Umberto Zannier}
\address{Scuola Normale Superiore, Piazza dei Cavalieri 7, 56126 Pisa, Italy}
\email{umberto.zannier@sns.it}


\begin{document}

\begin{abstract}
This note gives a simple algorithm for the following effectivity problem: given a genus $2$ curve $X$ together with a nonconstant map $\pi:X\to E$ to an elliptic curve, determine an elliptic curve $E'$ and a map $\pi':X\to E'$ independent of $\pi$. Equivalently, we compute the complementary elliptic factor in the decomposition of $\operatorname{Jac}(X)$ up to isogeny. While the problem has been studied extensively, and more general ones have been solved by deep and powerful techniques, we are not aware of a reference for the simple explicit procedure described here.
\end{abstract}

\maketitle

\section{Introduction}    

A natural problem in the explicit geometry of abelian varieties is to determine, effectively, a Poincar\'e decomposition of a given abelian variety $A$, that is, to describe $A$ as the product (up to isogeny) of simple abelian varieties. Closely related is the problem of deciding whether two given abelian varieties are isogenous. Even for elliptic curves, such questions can already be delicate, although in many cases non-isogeny is easy to establish, and effective criteria in the CM case are classical.

Effectivity in problems of this form is in general highly nontrivial. Foundational work of {\sc Masser} and {\sc W\"ustholz}~\cite{MasserWustholz1993,MasserWustholz1995} gives a powerful and general framework, in particular over number fields. These results have had many important applications, but they rely on deep methods and are usually not well suited for explicit computations in special cases. Presumably, the same circle of ideas should also yield a complete answer over any field finitely generated over $\mathbb Q$. Over function fields, one expects simpler and more direct arguments, but we are not aware of an explicit treatment of this case in the literature.

Recently, independent motivations led {\sc Gallese} and {\sc Naccarato} to the following special case of the problem of computing an isogeny decomposition:

\begin{problem}\label{prob}
Given a curve $X$ of genus $2$ with a rational nonconstant map $\pi:X\to E$ to an elliptic curve $E$, find an independent map $\pi':X\to E'$ to an elliptic curve $E'$.    
\end{problem}

\begin{remark}
Here \emph{independent} may be defined in several equivalent ways: for instance, by requiring that the image of $(\pi,\pi'):X\to E\times E'$ is not contained in a translate of a proper algebraic subgroup. This is automatic if $\pi'$ is nonconstant and $E,\, E'$ are not isogenous. We also note that the pair $E,\, E'$ is unique up to isogeny.
\end{remark}

Such a map $\pi'$ does in fact exist by the classical theory: if $J$ denotes the Jacobian of a smooth model of $X$, then $J$ is isogenous to a product $J \sim E\times E'$ of elliptic curves, and the composition $X \to J \to E'$ gives the desired map. Our objective here is to make $E'$ and $\pi'$ explicit.
Our aim is not to pursue generality, but rather to isolate in this setting a very simple and concrete procedure for recovering the complementary elliptic factor and the corresponding map.

We shall briefly recall below some of the motivations that led {\sc{Gallese}} and {\sc{Naccarato}} to consider this problem from an effective point of view. For the moment, we only note that, natural as the question may seem, we have not been able to find in the literature an explicit treatment in the form considered here; moreover, when asking colleagues about it, we were not pointed to a reference. Of course, related aspects of genus-$2$ curves with split Jacobians and maps to elliptic curves have been studied before: see, for instance,~\cite{Kuhn},~\cite{freyCurvesGenus21991},~\cite{Shaska}, and~\cite{MR3427148}. In contrast to the present work, where a map $X\to E$ is assumed to be given, the algorithmic problem of finding such a map (when it exists) is considered in~\cite{Lombardo}.

Since in the special genus-$2$ setting we found a very simple method, which also appears to be computationally efficient, we felt it worthwhile to record it. 
We would be surprised if the underlying principle were genuinely new, so the basic idea is probably already familiar to some experts, but an explicit account may still be of use to others.

In the sequel we first describe the method. We then collect several remarks and questions suggested by the construction, together with a few motivating examples, and finally we briefly discuss related problems and possible generalizations.

\section{The method}

We assume that we are given a map $\pi:X\to E$, with $X$, $E$, and $\pi$ all defined over a ground field $K$. We think of $K$ as a \emph{computable} field of characteristic different from 2, namely a field finitely generated over its prime field (say $\mathbb Q$), presented by finitely many explicit generators and algebraic relations. In practice, we assume that $X$ and $E$ are given by explicit equations, for instance as plane or space curves. We will take finite and computable extensions of the ground field $K$ as necessary (for example, to fix a $K$-rational base point on $X$) without further comment.

To keep the note as brief as possible, we shall not dwell on routine technical details, which can be supplied or modified without difficulty. For example, we will sometimes assume that $X$ is smooth and projective, although for actual computations it may be convenient to work with a singular or affine model.

\begin{notation}
    We suppose that $X$ is given by the affine equation 
\[
X : \, y^2=F(x)
\]
where $F\in K[x]$ has degree $5$ or $6$, and that $E$ is given by the Weierstrass equation 
\[
E: \, w^2=f(z),
\]
where $f(z) \in K[z]$ is monic of degree $3$. 
The map $\pi$ may then be written as
\[
\pi(x,y)=(Z(x,y), \, W(x,y)),
\]
where $Z$ and $W$ are rational functions over $K$.
\end{notation}

\begin{remark}
    Using the relation $y^2=F(x)$, we may write both $Z$ and $W$ as polynomials in $y$ of degree at most $1$ (with coefficients that are rational functions in $x$). They satisfy $W^2=f(Z)$ as functions on the curve $X$.
\end{remark}

\begin{remark}\label{rmk: standard form map}
Let $i$ denote the hyperelliptic involution on $X$, given in affine coordinates by $i(x,y)=(x,-y)$. Then the map $\pi+\pi\circ i$ factors through the quotient $X/i\cong \mathbb P^1$.
Since there is no nonconstant morphism from $\mathbb P^1$ to $E$, this map must be constant. After changing the origin on $E$ (or equivalently, post-composing with a translation on $E$), we may therefore assume that it is constantly equal to the identity of $E$, that is, $\pi\circ i=-\pi$. In terms of the affine representation $\pi(x,y)=(Z(x,y),\, W(x,y))$, this means that $Z$ and $W/y$ are rational functions of $x$.
In the following we thus write
\[ \pi(x,y)=(Z(x),\, yW(x)). \]
\end{remark}

\begin{remark}
    Whenever necessary, we may assume that $\pi\colon X \to E$ is \emph{primitive}, that is, it does not admit any proper intermediate subcover. One can always reduce to this case: factoring $X \to E$ as $X \to \tilde{E} \to E$ with $X \to \tilde{E}$ primitive amounts to finding the intermediate subfields between $K(E)$ and $K(X)$, which can be done by Galois theory. In the primitive case, there is a complementary map $\pi' : X \to E'$ 
    of degree equal to $\deg \pi$ \cite[p.~155]{freyCurvesGenus21991}; this is the map we will determine.
\end{remark}

\subsection{Coordinates on \texorpdfstring{$X^{(2)}$}{X^(2)}}
\label{section: coordinates}
Recall that the \emph{symmetric square} $X^{(2)}$ of $X$ is by definition the quotient of $X^2$ by the involution $(p_1, p_2) \mapsto (p_2, p_1)$. Most of our calculations will take place in $X^{(2)}$, so we now describe its function field. The function field $K(X^2)=K(x_1, x_2, y_1, y_2)$ is the algebraic extension of $K(x_1, x_2)$ with relations $y_1^2=F(x_1)$ and $y_2^2 = F(x_2)$.
It is a Galois extension of $K(x_1+x_2, x_1x_2)$ with group isomorphic to the dihedral group $D_4$, generated by
\[
\begin{aligned}
r \colon\;& \left|\,
\begin{array}{rcl}
(x_1,x_2) &\mapsto& (x_2,x_1),\\
(y_1,y_2) &\mapsto& (y_2,-y_1),
\end{array}
\right.
\qquad &
s \colon\;& \left|\,
\begin{array}{rcl}
(x_1,x_2) &\mapsto& (x_2,x_1),\\
(y_1,y_2) &\mapsto& (y_2,y_1).
\end{array}
\right.
\end{aligned}
\]
The field $K(X^{(2)})$ is the subfield of $K(X^2)$ fixed by $s$. In particular,
\[
    K(X^{(2)})=K(x_1+x_2,\, x_1x_2,\, y_1+y_2,\, y_1y_2).
\]
\vspace{-0.4cm}
\begin{notation}
We will use the four functions
\[
s_x := x_1+ x_2, \quad p_x := x_1x_2, \quad s_y := y_1+y_2, \quad p_y := y_1y_2
\]
as our system of coordinates on $X^{(2)}$.
\end{notation}

\subsection{The Jacobian of \texorpdfstring{$X$}{X}}
We view the Jacobian $J= \operatorname{Jac} X$ as birationally equivalent to $X^{(2)}$. Note that $X^{(2)}$ may be identified with the set of effective divisors of degree $2$ on $X$. Thus a point of $X^{(2)}$ may be written as
$(p_1)+(p_2)$, and corresponds to the class
$[(p_1)+(p_2)-2(p_0)]\in J$,
where $[\cdot]$ denotes linear equivalence of divisors and $p_0\in X$ is a fixed base point.
We may embed $X$ into $X^{(2)}$ by $p \longmapsto (p)+(p_0)$, and into $J$ by $x\longmapsto [(p)-(p_0)]$.

\begin{remark}
\label{remark: blow down}
    Notice that the map $X^{(2)} \to J$ need not be bijective: the preimage of a divisor class $[D] \in J$ consists of all pairs $(p_1) + (p_2) \in X^{(2)}$ such that $(p_1) + (p_2) - 2(p_0)$ is linearly equivalent to $D$. This preimage is naturally identified with the complete linear system $|D+2(p_0)|$, which is generically a single point (when $\dim |D+2(p_0)| = 0$) and isomorphic to $\mathbb{P}^1$ along a closed subvariety of $J$ (where $\dim |D+2(p_0)| = 1$)~\cite[\S 5]{MR861976}.
    Since two distinct degree-$2$ effective divisors on a genus $2$ curve are linearly equivalent only when they are in the canonical class, the morphism $X^{(2)} \to J$ is the blow-up at the point $[K_X-2(p_0)]$, where $[K_X]$ is the canonical class. If we take $p_0$ to be a Weierstrass point, $[K_X - 2(p_0)]$ is simply the origin of $J$.
\end{remark}

One may also describe the group law on $J$ in these terms: for generic points, it is expressed by linear equivalence of divisors, see for instance~\cite{Serre1988}.

\subsection{A model for the complementary curve \texorpdfstring{$E'$}{E'}}
Our approach is based on the following observation.
\begin{proposition}
    Let $\pi :X \to E$ be a primitive cover.
    For $c\in E$ a given point, consider the curve $\Phi_c$ in $X^{(2)}$ defined by the equation \[\pi(p_1)+\pi(p_2)=c, \] where $p_1, \, p_2\in X$. 
    The complementary curve $E'$ may be obtained as a smooth model of the unique irreducible component with positive genus. Moreover, for all but finitely many $c$, there is a unique component.
\end{proposition}
\begin{proof}
    Since $J$ is a blow-down of $X^{(2)}$ contracting a single copy of $\mathbb{P}^1$ (Remark \ref{remark: blow down}), the curve $\Phi_c$ is birational to a fiber of the map $\pi_* : J \to E$ (specifically, the one over the point $c-2\pi(p_0)$). Note that $\ker \pi_\ast \subseteq J$ is connected since $\pi$ is primitive~\cite[p.~154]{freyCurvesGenus21991}.
\end{proof}
\begin{remark}
    When $\pi$ is not primitive, $E'$ is birational to any positive-genus component of the curve $\pi(p_1)+\pi(p_2)=c$.
\end{remark}

\begin{remark}
\label{remark: trace computations}
The fiber over $c = 0$ is given by $\pi(p_1) = -\pi(p_2)$, so in the standard form of Remark \ref{rmk: standard form map}, this translates into the equations
$Z(x_1) = Z(x_2)$ and $y_1 W(x_1) + y_2 W(x_2) = 0.$
The first equation is divisible by $(x_1 - x_2)$, and this factor cuts out a component of $\Phi_0$ isomorphic to $\mathbb{P}^1$ that maps to a point in $E' \subseteq J$. This is the only genus-$0$ component of $\Phi_0$, provided that $\pi$ is primitive~\cite[\S 3]{gallese}. The equations for the complementary curve $\tilde{\Phi}_0$ are thus
\begin{equation}
    \label{eq: equations in the zero fiber}
    \frac{Z(x_1) - Z(x_2)}{x_1 - x_2} = 0 \qquad \text{and} \qquad y_1 W(x_1) + y_2 W(x_2) = 0 \quad \text{in}\,\, K(X^{(2)}).
\end{equation}
In particular, these equations define an irreducible curve.
\end{remark}

\subsection{The algorithm}\label{subsec: algorithm}

\begin{algorithm} 
\label{algorithm} Given a primitive cover $X \to E$, proceed as follows.

Step 1. Compute functions $Z(x)$ and $W(x)$ as in \Cref{rmk: standard form map}. 

Step 2. Express the equations for the complementary curve $\tilde{\Phi}_0$ in $X^{(2)}$ given in \Cref{remark: trace computations} in terms of the symmetric functions $x_1 + x_2, \, x_1 x_2, \, y_1 + y_2, \, y_1 y_2.$

Step 3. Compute a smooth curve $E'$ birational to $\tilde{\Phi}_0$ from these equations. 

Step 4. Compute the complementary map $\pi'$ (\Cref{algo: complementary map}).
\end{algorithm}

\begin{remark}
    This algorithm computes a subvariety $E'$ of $J$ up to birational equivalence. The construction will show that $E'$ is also isomorphic to the quotient $J/\pi^* E$. 
\end{remark}

We now describe each of these steps in greater detail.

\subsubsection{Step 1.}
Take an arbitrary point $p_0 \in X$. We compute $q_0 := \pi(p_0) + \pi(i(p_0)) \in E$ and a point $q_1 \in E$ such that $2q_1 = q_0$. Replacing $\pi$ with $\tau_{-q_1} \circ \pi$, where $\tau_{-q_1}$ is translation by $-q_1$, we may ensure that $\pi + \pi \circ i$ is constantly equal to $0 \in E$. By \Cref{rmk: standard form map}, $\pi$ is then of the form $(x,y) \mapsto (Z(x), \, yW(x))$.

\subsubsection{Step 2.}
\Cref{remark: trace computations} gives the equations of the complementary genus-$1$ curve $\tilde{\Phi}_0$ inside $X^{(2)}$. These equations are expressed in the variables $x_1, x_2, y_1, y_2$ and are fixed by the Galois automorphism $s$, so they belong to the subfield $K(X^{(2)})$.
Thus, the only computational task to be solved in Step 2 is the following: given an element $g$ of the field $K(X^2)=K(x_1, x_2, y_1, y_2)$, explicitly represented as a rational function in $x_1, x_2, y_1, y_2$ and known to lie in the subfield $K(X^{(2)})$, express $g$ as a rational function of $s_x, p_x, s_y, p_y$. Note that elements of $K(X^2)$ may be uniquely represented as $K(x_1,x_2)$-linear combinations of $1, y_1, y_2, y_1y_2$.

Decomposing $g$ along this basis and using its stability under the Galois automorphism $s$, we may write $g= g_0(x_1, x_2) + g_1(x_1,x_2) \, y_1 + g_1(x_2, x_1) \, y_2 + g_2(x_1, x_2) \, y_1y_2$ with $g_0(x_1, x_2)$ and $g_2(x_1, x_2)$ symmetric in their two variables. By taking iterated traces in function fields, one finds the identity
\begin{equation}\label{eq: representing g}
g_1(x_1,x_2) \, y_1 + g_1(x_2, x_1) \, y_2 = h_1(x_1, x_2) (y_1+y_2) - h_2(x_1, x_2) \, (y_1y_2)(y_1+y_2),
\end{equation}
where the functions
\[
h_1(x_1,x_2) = \frac{g_1(x_1,x_2) \, F(x_1) - g_1(x_2, x_1) \, F(x_2)}{F(x_1)-F(x_2)}, \quad h_2(x_1, x_2)=\frac{ g_1(x_1,x_2)  - g_1(x_2, x_1) }{F(x_1)-F(x_2)}
\]
are symmetric in $x_1, x_2$. This reduces our problem to representing the symmetric functions $g_0, g_2, h_1, h_2$ as rational functions of $x_1+x_2, x_1x_2$. This may be done by standard methods.

\subsubsection{Step 3.}
Since $F(x_1)+F(x_2)$ and $F(x_1)F(x_2)$ are symmetric functions in $x_1, x_2$, there exist computable rational functions $F_{\operatorname{sym}, +}$ and $F_{\operatorname{sym}, \cdot}$ such that
\[
F(x_1)+F(x_2) = F_{\operatorname{sym}, +}(s_x, p_x), \quad F(x_1)F(x_2) = F_{\operatorname{sym}, \cdot}(s_x, p_x).
\]
The function field of $X^{(2)}$ is generated by $s_x, p_x, s_y, p_y$, subject to the relations $s_y^2 = y_1^2 + 2y_1y_2 + y_2^2 = F(x_1)+F(x_2)+2p_y$ and $p_y^2=y_1^2y_2^2=F(x_1)F(x_2)$, that is,
\begin{eqnarray}
    p_y = \frac{s_y^2 - F_{\operatorname{sym}, +}(s_x, p_x)}{2}, \label{eq: eliminate py} \\
p_y^2 = F_{\operatorname{sym}, \cdot}(s_x, p_x). \label{eq: py2}
\end{eqnarray}
Thus, our model for $\tilde{\Phi}_0$ is cut out in the four-dimensional affine space in the variables $s_x, p_x, s_y, p_y$ by equations \eqref{eq: eliminate py} and \eqref{eq: py2}, together with
\begin{eqnarray}
\frac{Z(x_1)-Z(x_2)}{x_1-x_2}=0, \label{eq: P1} \\
y_1W(x_1) + y_2W(x_2)=0. \label{eq: select component}
\end{eqnarray}
By Step 2, we know how to express \eqref{eq: P1} and \eqref{eq: select component} in terms of $s_x, p_x, s_y, p_y$. 
Next, we show that the curve defined in the $(s_x, p_x)$-plane by \eqref{eq: P1} has geometric genus $0$.

\begin{remark}
    Although we write \eqref{eq: P1} and \eqref{eq: select component} as functions in $x_1, \, x_2, \, y_1, \, y_2$, one should see them as functions in $s_x, \, p_x, \, s_y, \, p_y$. Formally, one starts with the variety defined by \eqref{eq: P1} and \eqref{eq: select component} in $X^2$ and considers its image in $X^{(2)}$.

    Similarly, we may start with the curve defined by equation \eqref{eq: P1} in $(\mathbb{P}^1)^2$ and consider its image in the symmetric quotient $(\mathbb{P}^1)^{(2)}$. We denote this curve by $T$. This is the object we informally referred to as the curve defined by \eqref{eq: P1} in the $(s_x, p_x)$-plane.
\end{remark}

\begin{lemma}
    The curve $T$ has geometric genus $0$.
\end{lemma}
\begin{proof}
We denote by $j$ the involution of $X^{(2)}$ given by
\[
j\bigl((p_1)+(p_2)\bigr)=(i(p_1))+(i(p_2)).
\]

There is a map $\tilde{\Phi}_0 \to T$ induced by $k[s_x, p_x] \hookrightarrow k[s_x, p_x, s_y, p_y]$. 
This is clearly dominant and $j$-invariant.
It follows that $T$ is dominated by $\tilde{\Phi}_0/\langle j\rangle$. 

Under the birational identification $X^{(2)}\rightarrow J_X$, taking, without loss of generality, the base point $p_0$ to be a Weierstrass point of $X$, the involution $j$ corresponds to multiplication by $-1$ on $J$. Since $\tilde{\Phi}_0$ is birational to $E'$, we have
$\tilde{\Phi}_0/\langle j\rangle \sim_{\mathrm{bir}} E'/[-1]\cong \mathbb P^1$. In particular, $T$ is dominated by the genus 0 curve $\tilde{\Phi}_0/\langle j\rangle$ and therefore has genus $0$.
\end{proof}

\begin{remark}
    The polynomials $Z(x)$ for which the curve $\frac{Z(x_1)-Z(x_2)}{x_1-x_2}=0$ has a component of genus $0$ have been classified in \cite{MR2041613}. Note however that, as already pointed out, we do not directly work with \eqref{eq: P1}, but rather its degree-$2$ quotient $T$, so all we can say in our setting is that the curve in the $(x_1,x_2)$-plane defined by \eqref{eq: P1} is hyperelliptic (usually of positive genus: see for example Lemma \ref{lemma: critical EC}). Note furthermore that $Z(x)$ is usually a rational function and not a polynomial: it is a polynomial only if all the points in $\pi^{-1}(0_E)$ are points at infinity on $X$.
\end{remark}
    
    Since the curve $T$ deduced from \eqref{eq: P1} is of geometric genus $0$, we can parametrize $s_x,\, p_x$ in terms of a new variable $t$. Moreover, by \eqref{eq: eliminate py}, the variable $p_y$ is determined by the others. Eliminating $s_x,\, p_x,\, p_y$ in favor of $t$ and $s_y$ in \eqref{eq: py2} and \eqref{eq: select component} leaves us with two (reducible) curves in the $(t, s_y)$-plane. The only common component, which we may find by computing the GCD of the defining polynomials, gives a model for $\tilde{\Phi}_0$. 
    \begin{remark}\label{rmk: form of the model}
    This model is of the form $s_y^2=q(t)$ for some rational function $q(t)$. To see this, note first that the identities in Step 2 show that \eqref{eq: select component} may be written as $\left(H_1(s_x, p_x)  + H_2(s_x, p_x) \, p_y\right) \, s_y=0$ for suitable functions $H_1, H_2$. The factor $s_y=0$ does not appear in \eqref{eq: py2}, so when taking the GCD of the two remaining equations we obtain a divisor of $H_1(s_x, p_x)  + H_2(s_x, p_x) \, p_y$, which is in particular linear in $p_y$. Using \eqref{eq: eliminate py} to eliminate $p_y$, we obtain a polynomial of the form $H_1'(s_x, p_x) s_y^2 - H_2'(s_x,p_x)$, and finally replacing $s_x, p_x$ with their expressions in terms of $t$ leads to $s_y^2 = \frac{H_1'(s_x(t), p_x(t))}{H_2'(s_x(t), p_x(t))} =: q(t)$, as claimed.
    \end{remark}
    It only remains to compute the normalization of the curve $s_y^2 = q(t)$, which is easy, because we simply have to replace $q(t)$ with the product of the squarefree parts of its numerator and denominator.

\subsubsection{Step 4.}
It remains to compute the complementary map $\pi' : X \to E'$. The next proposition describes explicitly the natural map $X \hookrightarrow J \to J/\pi^*E \cong E'$.

\begin{proposition}
\label{prop: complementary map}
    Fix $c=0$ to be the origin of $E$, and identify $E'$ with the normalization of the genus-1 component $\tilde{\Phi}_0$ of $\Phi_0$. Choose any origin on $E'$ to make it an elliptic curve. A non-constant map $\pi':X\to E'$ is given by
    \[
        \pi'(p) =
        \sum_{\substack{q \in X \\ \pi(q)+\pi(p)=0 \\ q \neq i(p)}} ((p)+(q)) \quad
        - \sum_{\substack{q' \in X \\ \pi(q')+\pi(p_0)=0 \\ q' \neq i(p_0)}} ((p_0)+(q')),
    \]
    where every pair $(p)+(q)$ represents a point in $\tilde{\Phi}_0 \subseteq X^{(2)}$ and the sum is induced by the elliptic curve operation on $E'$. 
\end{proposition}
\begin{remark}
    The advantage of this statement over the abstract composition $X \hookrightarrow J \to J/\pi^*E \cong E'$ is that the description of \Cref{prop: complementary map} allows us to only take sums with respect to the elliptic curve structure on $E'$, rather than the more complicated group structure on $J$. Also note that, in the above sums, points $q$ in $\pi^{-1}(-\pi(p))$ and points $q'$ in $\pi^{-1}(-\pi(p_0))$ should be repeated according to their multiplicity.
\end{remark}
\begin{remark}
    In the statement of the proposition, we fix an arbitrary origin on $E'$. Since the two divisors being subtracted have the same degree, the resulting point of $\operatorname{Pic}^0(E')$ is independent of this auxiliary choice.
\end{remark}
\begin{proof}
    Through the natural embedding $X \hookrightarrow X^{(2)}$ and the birational morphism $X^{(2)} \to J$, a point $p \in X$ is sent to
    \begin{equation}
        \label{eq: X into J}
        p \mapsto (p)+(p_0) \mapsto (p)-(p_0)
    \end{equation}
    which lies on the fiber $\Phi_a$ with $a = \pi(p)+\pi(p_0)$. To get a point on the fixed fiber $\Phi_0$, 
    we compose \eqref{eq: X into J} with
    \begin{equation}
    \label{eq: translation}
        (p)-(p_0) \mapsto \deg(\pi)\cdot [(p)-(p_0)] - \pi^\ast\pi_\ast[(p)-(p_0)].
    \end{equation}
    The point on the right-hand side of \eqref{eq: translation} lies in $\ker \pi_\ast$, which we identify birationally with the genus-$1$ component $\tilde{\Phi}_0$ of $\Phi_0$.
    Moreover, if the resulting map $X \to \Phi_0$ were constant (hence constantly equal to $0$), every point of $X \subseteq J$ would lie in $[\deg \pi]^{-1}\left(\pi^\ast E\right)\subseteq J$. Since each irreducible component of $[\deg \pi]^{-1}\left(\pi^\ast E\right)$ is isomorphic to $\pi^* E \cong E$, this would force $X$ to be contained in an elliptic curve, contradiction.
    Expanding the definition of pull-back and push-forward, we rewrite the right-hand side of \eqref{eq: translation} as
    \[
        \deg(\pi)(p)-\deg(\pi)(p_0)
        - \sum_{\substack{q \in X \\ \pi(q)=\pi(p)}} (q) \quad
        + \sum_{\substack{q' \in X \\ \pi(q')=\pi(p_0)}}(q').
    \] 
    Let $\kappa$ be a canonical divisor on $X$. It is well known that $(q) \sim \kappa - (i(q))$ for every $q \in X$; using this for all $q$ (resp.~$q'$) that appear in the sum, and cancelling out all the occurrences of $\kappa$ (note that $\#\{ q: \pi(q) = \pi(p)\} = \#\{ q' : \pi(q') = \pi(p_0) \} = \deg \pi$, where the points are counted with multiplicity), we obtain
    \[ 
        \deg(\pi)(p)-\deg(\pi)(p_0) 
        + \sum_{\substack{q \in X \\ \pi(q)+\pi(p)=0}} (q) \quad
        - \sum_{\substack{q' \in X \\ \pi(q')+\pi(p_0)=0}} (q').
    \]
    Reorganizing, one gets the sum in the statement. The divisor $(p)+(i(p))-(p_0)-(i(p_0))$ is linearly equivalent to zero and therefore $(p)+(i(p)), \ (p_0)+(i(p_0))$ may be removed from the sums (one can show that $(q)+(i(q))$ lies in the genus-0 component of $\Phi_0$).
\end{proof}

We may compute the image of any point on the genus-$2$ curve $X$ under the map $\pi' : X \to E'$ using \Cref{prop: complementary map}. One can avoid the computation of the individual points $q$ in $\pi^{-1}(-\pi(p))$ by describing the full divisor
\begin{equation}\label{eq: divisor on Phi0tilde}
   \Gamma_p := \sum_{\substack{q \in X \\ \pi(q)+\pi(p)=0 \\ q \neq i(p)}} ((p)+(q))
\end{equation}
on the curve $\tilde{\Phi}_0$ as the vanishing locus of polynomials with coefficients in the field of definition of $p$. %
\begin{remark}
    Once the divisors $\Gamma_p,\, \Gamma_{p_0}$ are known, we obtain $\pi'(p)$ by summing the points in the support of $\Gamma_p - \Gamma_{p_0}$ using the (birational) group law on $\tilde{\Phi}_0$ (which we may compute by mapping the divisor to $E'$ and summing there, or by suitably applying the Riemann-Roch theorem, which avoids the computation of the points in its support).
\end{remark}

We now give equations for \eqref{eq: divisor on Phi0tilde}. Note that the conditions $\pi(q) + \pi(p)=0,\, q \neq i(p)$ simply amount to $(p)+(q)$ being a point on $\tilde{\Phi}_0$. Thus, \eqref{eq: divisor on Phi0tilde} is the divisor on $\tilde{\Phi}_0$ consisting of those $D = (p_1)+(p_2)$, seen as effective divisors of degree $2$, where one of the two points $p_1,\,  p_2$ in the support is the fixed point $p$. The set of divisors $(p)+(p_2)$ in $X^{(2)}$ for which $p$ is one of the two points in the support is the image in $X^{(2)}$ of
\begin{equation}\label{eq: two coordinate axes}
    \{p\} \times X \; \cup \; X \times \{p\} \subset X^2.
\end{equation}
The subvariety cut out in $X^2$ by the (symmetric) equations
\begin{equation}\label{eq: one point is p}
\begin{cases}
(x_1 - x(p)) (x_2 - x(p)) =0 \\ (y_1 -y(p))(y_2 -y(p))=0 
\end{cases}
\end{equation}
differs from \eqref{eq: two coordinate axes} only by a finite number of points: a point $(p_1, p_2)$ that satisfies the first equation has either $x_1=x(p)$ or $x_2=x(p)$. By symmetry, suppose we are in the first case. Then either $y_1 = y(p)$, in which case we have $p_1=p$ and hence $(p_1, p_2) \in \{p\} \times X$, or $y_1 \neq y(p)$, which forces $y_1=-y(p)$ and $y_2=y(p)$, which in turn gives that $x_2$ satisfies $F(x_2) = y_2^2 = y(p)^2$, giving only a finite number of points. Thus, for generic $p$, the projection of \eqref{eq: one point is p} to $X^{(2)}$ will meet $\tilde{\Phi}_0$ precisely in $\Gamma_p$. It is also easy to test if this condition is satisfied for a specific $p$, since this happens if and only if the intersection consists of precisely $\deg \Gamma_p = \deg \pi - 1$ points (counted with multiplicity).

Proceeding as in Step 2, we may write down functions on $X^{(2)}$ that cut out the locus described by \eqref{eq: one point is p}. Intersecting with $\tilde{\Phi}_0$ gives the desired divisor \eqref{eq: divisor on Phi0tilde}. This leads to the following algorithm for the computation of $\pi'$.

\begin{algorithm}\label{algo: complementary map}
    Given a primitive cover $\pi : X \to E$,
    \begin{enumerate}
        \item Compute $\tilde{\Phi}_0$, a smooth model $E' : (w')^2 = h(z')$, and the birational map $\tilde{\Phi}_0 \dashrightarrow E'$.
        \item 
        Compute the divisors $\Gamma_{i(p_0)}$ and $\Gamma_{p_0}$ on $\tilde{\Phi}_0$. Identify them with divisors on $E'$ via the birational map of step (1). Let $q_0$ be the point on $E'$ corresponding to the divisor class $[\Gamma_{i(p_0)}-\Gamma_{p_0}] \in \operatorname{Jac}(E')$,
        and let $q_1 \in E'$ be such that $q_0=2q_1$. 
        \item The map $\tau_{-q_1} \circ \pi'$ is computable and, by \Cref{rmk: standard form map}, of the form $(Z'(x), y W'(x))$.
        Compute $(Z'(x), y W'(x))$ for $2\deg\pi + 1$ points $(x,y) \in X$.
        \item Recover the rational function $Z'(x)$ by interpolation over these images.
        \item Deduce $W'(x) = \pm \sqrt{ h(Z'(x)) / F(x) }$.
    \end{enumerate}
\end{algorithm}

\begin{remark}
    In Step (2), it may happen that some points in the support of $\Gamma_{p_0}$ or $\Gamma_{i(p_0)}$ lie in the (finite) indeterminacy locus of the birational map $\tilde{\Phi}_0 \to E'$. If this is the case, one may simply use a different base point $p_0$.
\end{remark}

\begin{remark}
Alternatively, we could replace steps (3)-(5) with a single computation for a generic point of $X$, but this seems harder to implement. Also note that in Step (5) the sign may be fixed by comparison with one of the values computed in Step (3) -- using any such value for which $yW'(x) \neq 0$.
\end{remark}
\begin{proof}
    We have already justified the correctness and computability of all the steps except (5). Note that the degree of $Z'$ as a rational function in $x$ is equal to $\deg \pi'$, which in turn is equal to $\deg \pi$ since $\pi$ is primitive~\cite[p.~155]{freyCurvesGenus21991}. A degree-$n$ rational function may be interpolated from its values at $2n+1$ points.
\end{proof}

\section{Motivation}
\subsection{Higher dimensional split Jacobians}
The following natural question remains open: let $\pi \colon X \to X_1$ be a ramified cover such that $\operatorname{Jac} X$ splits, up to isogeny, as a product of Jacobians of curves $X_1, X_2, \ldots, X_r$. How can one explicitly recover the curves $X_i$ occurring in this decomposition?


Work of {\sc{Lombardo, Lorenzo Garc\'{\i}a, Ritzenthaler}}, and {\sc{Sijsling}}~\cite{MR4574430} suggests 
that the $X_i$ can be recovered as quotients of the Galois closure of the composition $X \to X_1 \to \mathbb{P}^1$ for a suitably chosen low-degree map $X_1 \to \mathbb{P}^1$. This strategy has been formalized into a conjecture and established 
in the case $g(X)=2,\, g(X_1)=1$ by \textsc{Gallese}~\cite{gallese}, via an explicit construction 
of the complementary elliptic curve $E' = X_2$. That construction has since been revised to align 
with the one presented in this note, substantially simplifying the arguments in the first version of~\cite{gallese}. 

\subsection{Critical values of polynomials}\label{subs: cv}

A different motivation arises from the third author's study of \textit{critical elliptic curves}~\cite[\S 2]{nac2}. These are, up to a twist, elliptic curves $E/K$ with a double cover $E\to\mathbb P^1$ that branches over the same four places as the cover $g:\mathbb P^1\to\mathbb  P^1$ induced by a quartic polynomial $g\in K[x]$. More specifically, given any quartic polynomial $g$, \textsc{Naccarato} defines the \textit{critical curve} relative to $g$ as the projective cubic with affine model
\begin{align}
 E_g: w^2=\operatorname{Disc}_x(g(x)-z).\label{eq: crEC}
 \end{align}
Here we assume $\mathrm{char}(K)=0$. We remark that \eqref{eq: crEC} defines an elliptic curve if and only if the $j$-invariant of the ramification points---that is, that of $u^2=g'(x)$---is different from $1728$ and $\infty$ (see~\cite[\S 4.2]{nac1}), which we assume from now on. The study of critical elliptic curves plays a central role in understanding the arithmetic of critical values (i.e.~branch points) of quartic polynomials; more details can be found in~\cite[\S 5]{nac1} and~\cite{nac2}.

The fiber product 
\begin{equation}\label{eq: polynomial genus 2}
    C_g=\mathbb P^1_x\underset{\mathbb P^1_z}{\times} E_g
\end{equation}
has the affine model
\begin{equation}
y^2=F(x):=\operatorname{Disc}_t\left(\frac{g(t)-g(x)}{t-x}\right).\label{eq: fp reduced}
\end{equation}
Indeed, let $\Delta(z)=\operatorname{Disc}_t(g(t)-z)$; \eqref{eq: polynomial genus 2} gives the equation \begin{align}\label{eq: naive fp model}
    w^2=\Delta(g(x))
\end{align} for $C_g$. Noticing that
\[
\Delta(g(x)) = \operatorname{Disc}_t(g(t)-g(x)) = \operatorname{Disc}_t
\left( (t-x)\frac{g(t)-g(x)}{t-x} \right)=g'(x)^2\operatorname{Disc}_t\left(\frac{g(t)-g(x)}{t-x}\right)
\]
and setting $y=(g'(x))^{-1}w$ transforms \eqref{eq: naive fp model} to $y^2=F(x)$.
This model is smooth: to show this, it suffices to prove that $F(x)$ has six distinct roots. Let $c_1, c_2, c_3$ be the three (distinct, by our previous assumption) finite branch points of $g$. The roots of $F$ are the $g$-preimages of its finite branch points where $g$ does not ramify, that is, the points $\{ x \in g^{-1}(c_i) \bigm\vert g'(x) \neq 0 \}$. As each root of $g'$ has to map with ramification index $2$ to a distinct branch point (a higher index would correspond to a double root of $g'$, preventing it from having two more roots), we get a total of $2+2+2=6$ distinct roots for $F$.

Observe that, by construction, there is a map $\pi_g:C_g\to E_g$ given by:
\begin{align}\label{pi deg 4}
     (x,y) \mapsto (g(x),yg'(x)).
\end{align}

A peculiarity of the quartic polynomial case is that one can drop the second equation from \eqref{eq: equations in the zero fiber}, once \eqref{eq: P1} is considered in $K(X^2)$ (rather than in the symmetric square):

\begin{lemma}
\label{lemma: critical EC}
    Let $g\in K[x]$ be a quartic polynomial such that \eqref{eq: crEC} defines an elliptic curve, and let $$X=C_g, \quad E=E_g, \quad \pi=\pi_g,$$ so that $Z(x)=g(x)$. Then, a plane affine model for $E'$ is given by \begin{align}
          \frac{Z(u) - Z(v)}{u - v} = 0.\label{eq:self-fiber product}
    \end{align}
\end{lemma}

\begin{proof}
Let $L$ be the splitting field of $g(x)-z$ over $K(z)$, the roots being $x=x_1,x_2,x_3,x_4$, and let $D$ be the normalization of the plane curve defined by \eqref{eq:self-fiber product}. It is well-known \cite[\S 4.4]{Serre1992} that if $E_g$ is an elliptic curve---the case of \textit{Morse} polynomials---then $G=\operatorname{Gal}(L/K(z))\simeq S_4$. Sending $(u,v)$ to $(x_1,x_2)$ gives an isomorphism between $K(D)$ and $L^{\langle(34)\rangle}$.
In~\cite[\S 4.1]{nac2} it is shown that $D$ has genus $1$. 

For each $j\in\{1,2,3,4\}$ we define the $L$-point
\[
P_j=\left(x_j,\frac{w}{g'(x_j)}\right)\in X(L),
\]
which satisfies $\pi(P_j) = (z, w)$; notice that $w = yg'(x_1)$. We also introduce
the degree $2$ effective divisor $H=P_3+i(P_4)\in X^{(2)}(L)$. Since, for $\sigma\in G, \ \sigma(w)=\operatorname{sgn}(\sigma)w$, one checks easily that $\operatorname{Stab}_G(H)=\langle(34)\rangle$. It follows that $H$ descends to an effective divisor of degree $2$ with coordinates in $L^{\langle (34)\rangle} \cong K(D)$, that is, a point in $X^{(2)}(K(D))$. By definition, this is the same as a rational map $D \dashrightarrow X^{(2)}$. Composing with the natural map $X^{(2)} \to J$ given by $(p_1)+(p_2) \mapsto [(p_1)+(p_2) - K_X]$, we finally get a rational map 
\[
\Psi : D \dashrightarrow J
\] which, since $D$ is a smooth curve and $J$ is projective, can in fact be extended to a morphism. By construction, the generic point of $D$ maps to $\eta := [(P_3) + (i(P_4)) - K_X]$. Applying $\pi_*$ to this point we find $\pi_*(\eta) = [(z,w) + (z,-w) - \pi_*(K_X)]=0$, so the generic point of $D$ maps inside $\ker \pi_*$, and therefore $\Psi(D) \subset E'$ .

To conclude that $\Psi$ induces a birational morphism $\Psi: D \to E'$, we just need to show that $\deg\Psi=1$. Two degree $2$ effective divisors on a genus $2$ curve are linearly equivalent only when they are in the canonical class, so on the open set $x_3\neq x_4$ we have injectivity, and we are done.
\end{proof}

\begin{remark}
    Let $Y$ be a curve with function field $L$. The proof of \Cref{lemma: critical EC} constructs the following diagram:
    
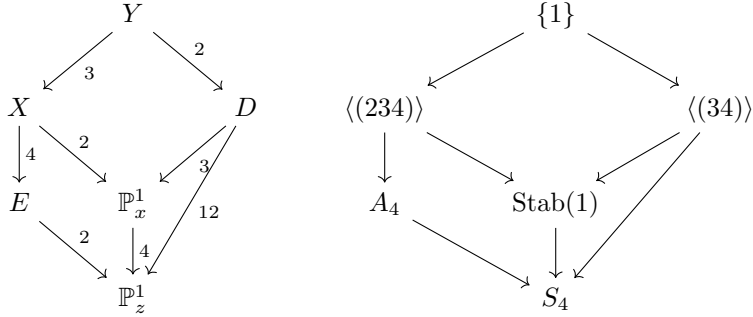
\[
\begin{tikzcd}
                                 & Y \arrow[ld, "3"] \arrow[rd, "2"] &               &                                           & \{1\} \arrow[ld] \arrow[rd]          &                                             \\
X \arrow[rd, "2"] \arrow[d, "4"] &                                            & D \arrow[ld, "3"] \arrow[ldd, "12"] & \langle(234) \rangle \arrow[rd] \arrow[d] &                                  & \langle (34) \rangle \arrow[ld] \arrow[ldd] \\
E \arrow[rd, "2"]                & \mathbb P^1_x \arrow[d, "4"]                 &               & A_4 \arrow[rd]                              & \operatorname{Stab}(1) \arrow[d] &                                             \\
                                 & \mathbb P^1_z                                &               &                                           & S_4                              &                                            
\end{tikzcd} \]

    Note that the extension $L / K(\mathbb{P}^1_x)$ is the splitting field of the cubic polynomial $(g(t)-g(x))/(t-x)$. The elliptic curve $E$ corresponds by definition to the quadratic extension of $K(z)$ generated by the square root of the discriminant of $g(x)-z$, and hence to the index-2 subgroup $A_4$ of $S_4$.
    The fiber product $X$ thus corresponds to the intersection of $A_4$ and the stabilizer of $1$ in $S_4$, namely, the group generated by $(234)$. 
    This is the copy of $A_3$ inside the group $S_3 \cong \operatorname{Gal}(Y/\mathbb{P}^1_x)$.
    Hence, $X$ corresponds to the quadratic extension of $K(x)$ generated by the square root of the discriminant of $(g(t)-g(x))/(t-x)$: this gives another proof that a model for $X$ is given by $y^2=F(x)$. Also note that $\pi$ is primitive, since $\langle(234)\rangle$ is not contained in any larger proper subgroup of $A_4$.
    Furthermore, one can prove that $\Psi\colon D \to J$ is induced by the correspondence $Y$ between $X$ and $D$~\cite[\S 2]{gallese}.
\end{remark}

Computing a model for the complementary curve $E'_g$ in $\operatorname{Jac}(C_g)$ has, for instance, the following application \cite[Theorem 2]{nac2}:
\begin{corollary}\label{cor: fp}
    Let $g$ be a quartic polynomial with rational coefficients. The equation
    \begin{align*}
        g(u)=g(t), \quad u\neq t
    \end{align*}
    has infinitely many solutions over $\mathbb Q(i)$, unless the $j$-invariant of the ramification points of $g$ is $0,\, -27648/11$ or $55296/5$.
\end{corollary}

\section{Examples}
\begin{example}[Degree 3]
\label{example: degree 3}
\textsc{Kuhn}~\cite[\S 6]{Kuhn} provides an explicit presentation of the moduli space of degree-$3$
covers $\pi \colon X \to E$ in terms of three parameters $a, b, c$. Taking $(a,b,c) = (3,4,5)$
yields the cover between the curves
\begin{align*}
    X &\colon y^2 = x^6 + \tfrac{19}{5}x^5 + \tfrac{42}{5}x^4 +
    \tfrac{309}{20}x^3 + \tfrac{63}{4}x^2 + 15x + \tfrac{25}{4} =: F(x), \\
    E &\colon -\frac{20}{247} w^2 = z^3 - \tfrac{84}{247}z^2 +
    \tfrac{164}{247}z - \tfrac{20}{247},
\end{align*}
given by $(x,y) \mapsto (Z(x), \, yW(x))$ with
\[
    Z(x) := \frac{x^2}{x^3 + 3x^2 + 4x + 5}, \qquad
    W(x)  :=  \frac{x^3 - 4x - 10}{x^6 + 6x^5 + 17x^4 + 34x^3 + 46x^2 + 40x + 25}.
\]
In this case, once expressed in terms of $s_x,\, p_x$, equation \eqref{eq: P1} becomes
\begin{equation}\label{eq: P1 example deg 3}
s_x - \frac{1}{5} p_x^2 + \frac{4}{5}p_x=0,
\end{equation}
which is clearly a curve of genus $0$ that we may parametrize using $p_x$. One can show in general that, for every degree-$3$ map $\pi$, the resulting equation \eqref{eq: P1} is linear in $s_x$. Once we eliminate $s_x$ and $p_y$ using \eqref{eq: P1 example deg 3} and \eqref{eq: eliminate py}, and take the GCD of the remaining equations in the variables $(p_x, s_y)$, we are left with the model
\[
5^6s_y^2 = p_x(p_x^3-4p_x^2+15p_x-25)(p_x^4 - 10p_x^3 + 32p_x^2 - \frac{189}{2}p_x + 150)^2,
\]
which has the form described in \Cref{rmk: form of the model}.
The normalization step is trivial, since it consists only of reabsorbing the squares, so we obtain the smooth genus one curve $y^2 = x(x^3-4x^2+15x-25)$, with Weierstrass model
\begin{equation}\label{eq: E prime example}
      E' :  y^2 = x^3 + 25x + 375.
\end{equation}
We verify that the $j$-invariant of $E'$ agrees with \textsc{Kuhn}'s computations, and confirm the same for a sample of further parameter triples $(a,b,c)$. Applying \Cref{algo: complementary map}, we find the explicit degree-$3$ map $\pi' : X \to E'$ given by
\[
\pi'(x,y) = \left( \frac{-6x^3 -6x^2 - 35/4x + 25/4}{x^3 + 4/5x^2 + 2x + 5/4}, \; \frac{-3x^3 + 55/2x^2 + 25/2x + 125/8 }{\left(x^3 + 4/5x^2 + 2x + 5/4\right)^2}\, y  \right).
\]
As predicted by~\cite{Kuhn}, the denominators of the $x$-coordinates of the maps $\pi$ and $\pi'$ satisfy $(x^3+3x^2+4x+5)(x^3+4/5x^2+2x+5/4) = F(x)$. We refer to~\cite{repo4} for the full computational details.

\end{example}
\begin{example}[Degree 4]
Taking as $Z(x)$ the quartic polynomial $g(x)=x^4-8x^2+16x$, \eqref{eq: crEC} and \eqref{eq: fp reduced} give models 
\[ E_g:-w^2=z^3 + 32z^2 - 896z + 4864, \quad C_g:-y^2=x^6 - 16x^4 + 40x^3 + 80x^2 - 288x + 304 \]
for the curves of \S\ref{subs: cv}, after rescaling $w$ and $y$. Exploiting Lemma \ref{lemma: critical EC}, we can compute the Weierstrass model $$(w')^2=(z')^3-432z'-8208$$ for $E'_g$. This is the quadratic twist by $-1$ of the elliptic curve with Cremona label 11a3~\cite{lmfdb}, the so-called \textit{first elliptic curve in nature}. Notice that $(E'_g)^{(-1)}(\mathbb{Q})\simeq\mathbb{Z}/5\mathbb{Z}$ is finite: this is one of the exceptional cases of Corollary \ref{cor: fp}, as one can check by verifying that the $j$-invariant of $u^2=g'(x)$ is $-\frac{27648}{11}$.
Indeed, generically we have $\mathrm{rk}(E'_g)^{(-1)}(\mathbb{Q})>0$, which is the key fact behind the infinitude result given in the corollary, see~\cite[\S 4]{nac2}.

\end{example}
\begin{example}[Degree 5]
    \label{example: degree 5}
We perform a computation analogous to that of~\Cref{example: degree 3}, starting from
\textsc{Shaska}'s presentation~\cite{Shaska} of the moduli space of degree-$5$ covers
$\pi \colon X \to E$ in terms of two parameters $a, b$. We take $(a,b) = (7,1)$.
Any model for $X$ is defined over an extension of $\mathbb{Q}$ of degree at least $2$; we take as base field
the quadratic number field $\mathbb{Q}(z)$, where $z^2 + z - 5 = 0$. Explicit models for $E$, $X$,
and the map $\pi$ can be found by running the code in the accompanying repository~\cite{repo4}.

The main difference from the degree-$3$ case arises in Equation \eqref{eq: P1}, which gives
\begin{align*}
    s_x^4 & + 130 s_x^3 p_x + 130 s_x^3
    + 4255 s_x^2 p_x^2 - 33728 s_x^2 p_x + 4255 s_x^2 \\
    &+ 19500 s_x p_x^3 + 114140 s_x p_x^2 + 114140 s_x p_x + 19500 s_x \\
    &+ 225 p_x^4 - 707130 p_x^3 + 14336251 p_x^2 - 707130 p_x + 225 = 0.
\end{align*}
This again defines a curve of geometric genus $0$, which, unlike the degree-$3$ case, admits no non-singular point over $\mathbb{Q}(z)$, so producing a parametrization requires a further field
extension. Carrying this out and substituting into the remaining equations, one obtains
a singular model for $E'$ whose resolution yields a Weierstrass equation with the
same $j$-invariant as one can obtain from the formulas in \cite{Shaska}.
\end{example}

\section{Further comments}

\subsection{On the \texorpdfstring{$j$}{j}-invariants of \texorpdfstring{$E,\, E'$}{E, E'}}
Assume $\pi$ is primitive of degree $n$. Then $\operatorname{Jac} X$ is isogenous to $(E\times E')/\Gamma_\psi$, where $\psi:E[n]\xrightarrow{\sim}E'[n]$ is a Galois-equivariant anti-isometry for the Weil pairing~\cite{freyCurvesGenus21991,MR1748483} and $\Gamma_\psi$ is its graph. If $E,\, E'$ are non-isogenous and $K$ is large enough (e.g.~contains $E[n],\, E'[n]$), such $\psi$ exist, and one can realize any such pair via a genus-$2$ curve with degree-$n$ maps~\cite[Proposition~1.5]{freyCurvesGenus21991}. In particular, we recover the well-known fact that, even for fixed $n$, there is no universal equation $P_n(j(E),j(E'))=0$ satisfied by the $j$-invariants of $E$ and $E'$. We give a simple example that demonstrates this explicitly for $n=2$. 

\begin{example}
Let $n=2$. For sufficiently general $j,\, j' \in \overline{K}$, we may find Weierstrass equations
\[
E:\ y^2=x^3+ax^2+bx+1 \qquad \text{and} \qquad E':\ y^2=x^3+bx^2+ax+1
\]
with $j(E)=j, \, j(E')=j'$. Then
\[
X:\ y^2=x^6+ax^4+bx^2+1
\]
admits degree-$2$ maps
$(x,y)\mapsto(x^2,y)$ and $(x,y)\mapsto(x^{-2},yx^{-3})$ to $E$ and $E'$, respectively. Since generically any pair $j(E),\, j(E')$ may be realized in this way, the $j$-invariants of $E,\, E'$ satisfy no algebraic relation.
\end{example}

\begin{remark}
    One may still wonder about families of covers $C \to E$ determined by an algebraic relation between the $j$-invariants of $E$ and $E'$. A collection of such families is provided by the modular polynomials \cite{MR2555986, MR4932442}, which prescribe an isogeny between $E$ and $E'$. Naturally interesting families arise by prescribing the ramification behaviour of the associated cover $g : C/\langle i \rangle \to E/[-1]$; for instance,
    in the context of Lemma \ref{lemma: critical EC},
    the elliptic curves $E_g$ and $E'_g$ are generically non-isogenous, but, letting $j$ be the $j$-invariant of $u^2=g'(x)$, we have:
    $$j(E)=\frac{j(j-1536)^3}{2^{18}(j-1728)}, \quad j(E')=\frac{j^2}{4(j-1728)}.$$
    For a proof of these facts, see~\cite[\S 5]{nac2}.
\end{remark}

\subsection{Connection with quadratic twists.}
Assume, as in Remark \ref{rmk: standard form map}, that $\pi\circ i=-\pi$. We may then write $\pi(x,y)=\bigl(Z(x),\,yW(x)\bigr)$ with
\[
F(x)W(x)^2 = f(Z(x)).
\]
Equivalently, $P(t)=\bigl(Z(t),W(t)\bigr)$ is a point in $E^{(F(t))}(K(t))$, where
$E^{(F(t))}:\ F(t)w^2=f(z)$ denotes the quadratic twist of $E$ by $F(t) \in K(t)^\times$. 
Such a point $P(t)$ is necessarily non-torsion: over $K(t,\sqrt{F(t)})$ the twist is isomorphic to $E$ via $(z,w)\mapsto (z,\sqrt{F(t)}\,w)$, and a torsion point would have coordinates that are constant in $t$, contradicting the non-constancy of $Z(t)$. Conversely, any nonconstant $K(t)$-point on $E^{(F(t))}$ yields a surjective map $X\to E$.
The complementary map $\pi' : X \to E'$ similarly yields a nonconstant point $P'(t)$ on
\[
E'^{(F(t))}:\ F(t){w'}^2=h(z').
\]
Specializing at $t=t_0\in K$ with $F(t_0)\neq 0$ gives points $P(t_0), P'(t_0)$ on $E^{(F(t_0))}$ and $E'^{(F(t_0))}$. When $K$ is a number field, by Silverman's specialization theorem~\cite{SilvermanSpecialization}, these are non-torsion for $t_0$ of large height, so we have proved that there exist infinitely many $d=F(t_0) \in K$ such that the quadratic twists $E^{(d)},\, E'^{(d)}$ both have positive rank. Faltings's theorem implies that for each $d \in K^\times$ there are only finitely many solutions $(t_0, y)$ to the equation $dy^2=F(t_0)$, so the $d$ we have constructed give infinitely many distinct classes in $K^\times / K^{\times 2}$.

\begin{remark}
    Stronger results are known: for instance,~\cite[Theorem 4]{MR1214702} shows that, for any pair of elliptic curves $E_1, E_2$ over $\mathbb{Q}$ that satisfy neither $j(E_1)=j(E_2)=0$ nor $j(E_1)=j(E_2)=1728$, there exist infinitely many rational numbers $d$ such that $E_1^{(d)}$ and $E_2^{(d)}$ both have positive rank. The proof relies on the consideration of the Kummer surface $f_1(x_1)=z^2 f_2(x_2)$, where $E_i$ is defined by the equation $y^2 = f_i(x)$. Similar ideas, also related to the density of rational points on Kummer surfaces, appear in the appendix of \cite{MR3648511}.
\end{remark}

\subsection{Generalizations} If $\pi : X \to E$ is a cover from a curve of arbitrary genus $g$ to an elliptic curve, we may proceed in a similar way to describe a complement of $\pi^*E$ inside $J:=\operatorname{Jac} X$. Namely, we may birationally identify $J$ with $X^{(g)}$, and describe birationally the Prym variety of $X \to E$ (that is, $\operatorname{Prym}(X/E) := (\ker \pi_* : J \to E)^0$) as an irreducible component of $\pi(p_1) + \cdots + \pi(p_g)=c$ for sufficiently general $c \in E$. When $g(X)=3$, $\operatorname{Prym}(X/E)$ is generically isogenous to the Jacobian of a genus-$2$ curve $Y$, but unfortunately, we do not know an easy way to recover a suitable $Y$ from this construction (see~\cite{MR3781951} and~\cite{MR4574430} for results in this direction).

Similarly, if $\pi : Y \to X$ is a cover between curves of genera $g(Y)=3$ and $g(X)=2$, one may use similar techniques to describe the elliptic curve $E := \operatorname{Prym}(Y/X)$ birationally inside $Y^{(3)}$. However, in this case, the Riemann-Hurwitz formula shows that $\pi$ is necessarily an unramified double cover, and there are easier ways to recover $E$ (see for example~\cite[Proposition 2.2]{MR2406115}).

\bibliographystyle{plain}
\bibliography{biblio}

@book{Serre1988,
  author    = {Serre, Jean-Pierre},
  title     = {Algebraic Groups and Class Fields},
  series    = {Graduate Texts in Mathematics},
  volume    = {117},
  publisher = {Springer},
  address   = {New York},
  year      = {1988},
  doi       = {10.1007/978-1-4612-1035-1}
}

@article {Kuhn,
    AUTHOR = {Kuhn, Robert M.},
     TITLE = {Curves of genus {$2$} with split {J}acobian},
   JOURNAL = {Trans. Amer. Math. Soc.},
  FJOURNAL = {Transactions of the American Mathematical Society},
    VOLUME = {307},
      YEAR = {1988},
    NUMBER = {1},
     PAGES = {41--49},
      ISSN = {0002-9947,1088-6850},
   MRCLASS = {14H40 (11G10)},
  MRNUMBER = {936803},
MRREVIEWER = {H.-G.\ R\"uck},
       DOI = {10.2307/2000749},
       URL = {https://doi.org/10.2307/2000749},
}

@article{MasserWustholz1993,
  author  = {David W. Masser and Gisbert W\"ustholz},
  title   = {Isogeny estimates for abelian varieties, and finiteness theorems},
  journal = {Annals of Mathematics},
  series  = {2},
  volume  = {137},
  number  = {3},
  pages   = {459--472},
  year    = {1993},
  doi     = {10.2307/2946529}
}

@article{MasserWustholz1995,
  author  = {David W. Masser and Gisbert W\"ustholz},
  title   = {Factorization estimates for abelian varieties},
  journal = {Publications Math\'ematiques de l'IH\'ES},
  volume  = {81},
  pages   = {5--24},
  year    = {1995},
  doi     = {10.1007/BF02699374}
}

@misc{lmfdb,
  Key    = {LMFDB},
  author       = {The {LMFDB Collaboration}},
  title        = {The {L}-functions and modular forms database},
  howpublished = {\url{https://www.lmfdb.org/EllipticCurve/Q/11a3/}},
}

@article {SilvermanSpecialization,
    AUTHOR = {Silverman, Joseph H.},
     TITLE = {Heights and the specialization map for families of abelian
              varieties},
   JOURNAL = {J. Reine Angew. Math.},
  FJOURNAL = {Journal f\"ur die Reine und Angewandte Mathematik. [Crelle's
              Journal]},
    VOLUME = {342},
      YEAR = {1983},
     PAGES = {197--211},
      ISSN = {0075-4102,1435-5345},
   MRCLASS = {14K15 (14D10 14G25)},
  MRNUMBER = {703488},
MRREVIEWER = {Gerd\ Faltings},
       DOI = {10.1515/crll.1983.342.197},
       URL = {https://doi.org/10.1515/crll.1983.342.197},
}

@article {Shaska,
    AUTHOR = {Shaska, Tanush},
     TITLE = {Curves of genus 2 with {$(N,N)$} decomposable {J}acobians},
   JOURNAL = {J. Symbolic Comput.},
  FJOURNAL = {Journal of Symbolic Computation},
    VOLUME = {31},
      YEAR = {2001},
    NUMBER = {5},
     PAGES = {603--617},
      ISSN = {0747-7171,1095-855X},
   MRCLASS = {14H45 (14H30 14H40)},
  MRNUMBER = {1828706},
       DOI = {10.1006/jsco.2001.0439},
       URL = {https://doi.org/10.1006/jsco.2001.0439},
}

@article {Lombardo,
    AUTHOR = {Lombardo, Davide},
     TITLE = {Computing the geometric endomorphism ring of a genus-2
              {J}acobian},
   JOURNAL = {Math. Comp.},
  FJOURNAL = {Mathematics of Computation},
    VOLUME = {88},
      YEAR = {2019},
    NUMBER = {316},
     PAGES = {889--929},
      ISSN = {0025-5718,1088-6842},
   MRCLASS = {11F80 (11G10 11Y99)},
  MRNUMBER = {3882288},
MRREVIEWER = {John\ L.\ Boxall},
       DOI = {10.1090/mcom/3358},
       URL = {https://doi.org/10.1090/mcom/3358},
}

@misc{gallese,
      title={How to split two-dimensional {J}acobians: a geometric construction}, 
      author={Andrea Gallese},
      year={2025},
      eprint={2412.07414},
      archivePrefix={arXiv},
      primaryClass={math.AG},
      note={\url{https://arxiv.org/abs/2412.07414}}, 
}

@unpublished{nac1,
      title={The arithmetic of critical values {I}: equicritical quartic polynomials}, 
      author={Francesco Naccarato},
      year={2024},
      archivePrefix={arXiv},
      note={\url{https://arxiv.org/abs/2501.03244}}}

@misc{nac2,
      title={The arithmetic of critical values {II}: critical elliptic curves and applications}, 
      author={Francesco Naccarato},
      year = {in preparation}
}

@book{Serre1992,
  author    = {Serre, Jean-Pierre},
  title     = {Topics in Galois Theory},
  series    = {Research Notes in Mathematics},
  publisher = {Jones \& Bartlett},
  address   = {Boston},
  year      = {1992}
}

@article {MR4574430,
    AUTHOR = {Lombardo, Davide and Lorenzo Garc\'{\i}a, Elisa and
              Ritzenthaler, Christophe and Sijsling, Jeroen},
     TITLE = {Decomposing {J}acobians via {G}alois covers},
   JOURNAL = {Exp. Math.},
  FJOURNAL = {Experimental Mathematics},
    VOLUME = {32},
      YEAR = {2023},
    NUMBER = {1},
     PAGES = {218--240},
      ISSN = {1058-6458,1944-950X},
   MRCLASS = {14H40 (11R32 14H45 14L30)},
  MRNUMBER = {4574430},
MRREVIEWER = {Daniel\ Valli\`eres},
       DOI = {10.1080/10586458.2021.1926008},
       URL = {https://doi.org/10.1080/10586458.2021.1926008},
}

@article {MR3427148,
    AUTHOR = {Kumar, Abhinav},
     TITLE = {Hilbert modular surfaces for square discriminants and elliptic
              subfields of genus 2 function fields},
   JOURNAL = {Res. Math. Sci.},
  FJOURNAL = {Research in the Mathematical Sciences},
    VOLUME = {2},
      YEAR = {2015},
     PAGES = {Art. 24, 46},
      ISSN = {2522-0144,2197-9847},
   MRCLASS = {11F41 (14H40 14J28)},
  MRNUMBER = {3427148},
MRREVIEWER = {Michael\ M.\ Schein},
       DOI = {10.1186/s40687-015-0042-9},
       URL = {https://doi.org/10.1186/s40687-015-0042-9},
}

@incollection {MR861976,
    AUTHOR = {Milne, James S.},
     TITLE = {Jacobian varieties},
 BOOKTITLE = {Arithmetic geometry ({S}torrs, {C}onn., 1984)},
     PAGES = {167--212},
 PUBLISHER = {Springer, New York},
      YEAR = {1986},
      ISBN = {0-387-96311-1},
   MRCLASS = {14H40},
  MRNUMBER = {861976},
}

@incollection{freyCurvesGenus21991,
    AUTHOR = {Frey, Gerhard and Kani, Ernst},
     TITLE = {Curves of genus {$2$} covering elliptic curves and an
              arithmetical application},
 BOOKTITLE = {Arithmetic algebraic geometry ({T}exel, 1989)},
    SERIES = {Progr. Math.},
    VOLUME = {89},
     PAGES = {153--176},
 PUBLISHER = {Birkh\"auser Boston, Boston, MA},
      YEAR = {1991},
      ISBN = {0-8176-3513-0},
   MRCLASS = {14G40 (11G05)},
  MRNUMBER = {1085258},
MRREVIEWER = {Joseph\ H.\ Silverman},
       DOI = {10.1007/978-1-4612-0457-2\_7},
       URL = {https://doi.org/10.1007/978-1-4612-0457-2_7},
}

@misc{repo4,
  author       = {Gallese, Andrea and Lombardo, Davide and Naccarato, Francesco and Zannier, Umberto},
  title        = {{GitHub repository}},
  year         = {2026},
  publisher    = {GitHub},
  howpublished = {\url{https://github.com/G4ll/NoteOnGenus2Covers}},
}

@article {MR1214702,
    AUTHOR = {Kuwata, Masato and Wang, Lan},
     TITLE = {Topology of rational points on isotrivial elliptic surfaces},
   JOURNAL = {Internat. Math. Res. Notices},
  FJOURNAL = {International Mathematics Research Notices},
      YEAR = {1993},
    NUMBER = {4},
     PAGES = {113--123},
      ISSN = {1073-7928,1687-0247},
   MRCLASS = {11G05 (14J27)},
  MRNUMBER = {1214702},
MRREVIEWER = {Henri\ Darmon},
       DOI = {10.1155/S107379289300011X},
       URL = {https://doi.org/10.1155/S107379289300011X},
}

@article {MR1748483,
    AUTHOR = {Howe, Everett W. and Lepr\'evost, Franck and Poonen, Bjorn},
     TITLE = {Large torsion subgroups of split {J}acobians of curves of
              genus two or three},
   JOURNAL = {Forum Math.},
  FJOURNAL = {Forum Mathematicum},
    VOLUME = {12},
      YEAR = {2000},
    NUMBER = {3},
     PAGES = {315--364},
      ISSN = {0933-7741,1435-5337},
   MRCLASS = {11G30 (14H25 14H40)},
  MRNUMBER = {1748483},
MRREVIEWER = {Edward\ F.\ Schaefer},
       DOI = {10.1515/form.2000.008},
       URL = {https://doi.org/10.1515/form.2000.008},
}

@article {MR3781951,
    AUTHOR = {Ritzenthaler, Christophe and Romagny, Matthieu},
     TITLE = {On the {P}rym variety of genus 3 covers of genus 1 curves},
   JOURNAL = {\'Epijournal G\'eom. Alg\'ebrique},
  FJOURNAL = {\'Epijournal de G\'eom\'etrie Alg\'ebrique. EPIGA},
    VOLUME = {2},
      YEAR = {2018},
     PAGES = {Art. 2, 8},
      ISSN = {2491-6765},
   MRCLASS = {14H40 (11G05 14K02 14Q05)},
  MRNUMBER = {3781951},
MRREVIEWER = {Jordi\ Gu\`ardia},
       DOI = {10.46298/epiga.2018.volume2.3663},
       URL = {https://doi.org/10.46298/epiga.2018.volume2.3663},
}

@article {MR2406115,
    AUTHOR = {Bruin, Nils},
     TITLE = {The arithmetic of {P}rym varieties in genus 3},
   JOURNAL = {Compos. Math.},
  FJOURNAL = {Compositio Mathematica},
    VOLUME = {144},
      YEAR = {2008},
    NUMBER = {2},
     PAGES = {317--338},
      ISSN = {0010-437X,1570-5846},
   MRCLASS = {11G30 (14H40)},
  MRNUMBER = {2406115},
MRREVIEWER = {Jordi\ Gu\`ardia},
       DOI = {10.1112/S0010437X07003314},
       URL = {https://doi.org/10.1112/S0010437X07003314},
}

@article {MR3648511,
    AUTHOR = {Corvaja, Pietro and Zannier, Umberto},
     TITLE = {On the {H}ilbert property and the fundamental group of
              algebraic varieties},
   JOURNAL = {Math. Z.},
  FJOURNAL = {Mathematische Zeitschrift},
    VOLUME = {286},
      YEAR = {2017},
    NUMBER = {1-2},
     PAGES = {579--602},
      ISSN = {0025-5874,1432-1823},
   MRCLASS = {14G05 (12E25 12F12 14J28)},
  MRNUMBER = {3648511},
MRREVIEWER = {Daniel\ Loughran},
       DOI = {10.1007/s00209-016-1775-x},
       URL = {https://doi.org/10.1007/s00209-016-1775-x},
}

@article {MR2041613,
    AUTHOR = {Avanzi, Roberto M. and Zannier, Umberto},
     TITLE = {The equation {$f(X)=f(Y)$} in rational functions {$X=X(t)$},
              {$Y=Y(t)$}},
   JOURNAL = {Compositio Math.},
  FJOURNAL = {Compositio Mathematica},
    VOLUME = {139},
      YEAR = {2003},
    NUMBER = {3},
     PAGES = {263--295},
      ISSN = {0010-437X,1570-5846},
   MRCLASS = {14H45 (11C08 11G30 11R09 12E05)},
  MRNUMBER = {2041613},
MRREVIEWER = {Yuri\ Bilu},
       DOI = {10.1023/B:COMP.0000018136.23898.65},
       URL = {https://doi.org/10.1023/B:COMP.0000018136.23898.65},
}

@incollection {MR2555986,
    AUTHOR = {Frey, Gerhard and Kani, Ernst},
     TITLE = {Curves of genus 2 with elliptic differentials and associated
              {H}urwitz spaces},
 BOOKTITLE = {Arithmetic, geometry, cryptography and coding theory},
    SERIES = {Contemp. Math.},
    VOLUME = {487},
     PAGES = {33--81},
 PUBLISHER = {Amer. Math. Soc., Providence, RI},
      YEAR = {2009},
      ISBN = {978-0-8218-4716-9},
   MRCLASS = {14G32 (11G30 14G25 14H30)},
  MRNUMBER = {2555986},
MRREVIEWER = {Meirav\ Amram},
       DOI = {10.1090/conm/487/09524},
       URL = {https://doi.org/10.1090/conm/487/09524},
}

@article {MR4932442,
    AUTHOR = {Djukanovi\'{c}, Martin},
     TITLE = {Families of split {J}acobians with isogenous components},
   JOURNAL = {J. Th\'{e}or. Nombres Bordeaux},
  FJOURNAL = {Journal de Th\'{e}orie des Nombres de Bordeaux},
    VOLUME = {37},
      YEAR = {2025},
    NUMBER = {1},
     PAGES = {49--77},
      ISSN = {1246-7405,2118-8572},
   MRCLASS = {14H40 (11G10 14H10 14H52 14K02)},
  MRNUMBER = {4932442},
       DOI = {10.5802/jtnb.1312},
       URL = {https://doi.org/10.5802/jtnb.1312},
}










\end{document}